\documentclass[10pt,a4paper]{article} 
\pdfoutput=1
\author{Valentin Plechinger}
\date{April 2018}
\usepackage[utf8]{inputenc}
\usepackage{amssymb,amsmath, amsthm}
\usepackage{mathtools}
\usepackage{amsfonts}
\usepackage{enumerate}
\usepackage{tikz}
\usepackage{graphicx}
\usetikzlibrary{matrix,arrows,shapes,trees}
\usetikzlibrary{decorations.pathmorphing,patterns}
\usetikzlibrary{calc,patterns,decorations.markings}
\usetikzlibrary{positioning}
\usepackage{fancybox}
\usepackage{tikz-cd}
\usepackage{hyperref}
\usepackage{ifthen}
\usepackage[english]{babel}
\usepackage{fancyhdr}
\usepackage{afterpage}
\usepackage{amsmath,calligra,mathrsfs}

\newtheorem{theorem}{Theorem}[section]
\newtheorem{lemma}[theorem]{Lemma}
\newtheorem{proposition}[theorem]{Proposition}
\newtheorem{corollary}[theorem]{Corollary}
\newtheorem{remark}[theorem]{Remark}
\newtheorem{definition}[theorem]{Definition}
\newtheorem{named-theorem}[theorem]{}
\newcounter{example}
\newenvironment{example}[1][]{\refstepcounter{example}\par\medskip\noindent%
  \textbf{Example. #1} \rmfamily}{\medskip}

\newcommand*{\N}{\mathbb{N}}
\newcommand*{\Z}{\mathbb{Z}}

\newcommand*{\C}{\mathbb{C}}

\newcommand*{\shf}[1]{\mathcal{O}_{#1}}

\newcommand*{\ra}{\rightarrow}
\newcommand*{\xra}{\xrightarrow}

\newcommand*{\coh}[3]{H^{#1}(#2,#3)}
\newcommand*{\cstar}{\mathbb{C}^*}

\newcommand*{\lb}{\mathcal{L}}

\newcommand*{\iso}{\cong}
\newcommand*{\cat}[1]{\mathbf{#1}}
\newcommand*{\An}{\cat{An}}
\newcommand*{\Vect}{\cat{Vec}}
\newcommand*{\Set}{\cat{Set}}
\newcommand*{\ff}{\mathcal{F}}
\newcommand*{\Tau}{\mathcal{T}}
\newcommand*{\poincare}{{\mathfrak L}_{x_0}}

\newcommand*{\id}{\mathrm{id}}

\DeclareMathOperator{\NS}{NS}
\DeclareMathOperator{\Pic}{Pic}
\DeclareMathOperator{\Picaff}{Picaff}

\DeclareMathOperator{\GL}{GL}

\DeclareMathOperator{\Aff}{Aff}
\DeclareMathOperator{\Affshf}{\mathcal{A}}
\DeclareMathOperator{\pr}{pr}
\DeclareMathOperator{\Ext}{Ext}

\DeclareMathOperator{\Hom}{Hom}

\DeclareMathOperator{\shom}{\mathscr{H}\text{\kern -3pt {\calligra\large om}}\,}
\newcommand*{\gaff}{\Aff(1,\C)}

\def\textmap#1{\mathop{\vbox{\ialign{
                                  ##\crcr
      ${\scriptstyle\hfil\;\;#1\;\;\hfil}$\crcr
      \noalign{\kern 1pt\nointerlineskip}
      \rightarrowfill\crcr}}\;}}
\def\resto#1#2{{
#1\hskip 0.4ex\vline_{\hskip 0.2ex\raisebox{-0,2ex}
{{${\scriptstyle #2}$}}}}}

\begin{document}
\title{Classifying affine line bundles on a compact complex space}
\maketitle
\begin{abstract}
The classification of affine line bundles on a compact complex space is a
difficult problem. We study the affine analogue of the Picard functor and the
representability problem for this functor. Let $X$ be a compact complex space
with $H^0({\cal O}_X)=\C$. Fix $c\in\NS(X)$, $x_0\in X$. We introduce the affine
Picard functor $\Picaff_{X,x_0}^c:\An^{op}\ra \Set$ which assigns to a complex space
$T$ the set of families of linearly $x_0$-framed affine line bundles on $X$ parameterized
by $T$. Our main result states that the functor $\Picaff_{X,x_0}^c$ is
representable if and only if the   map $h^0:\Pic^c(X)\to\N$ is constant. If this
is the case, the space which represents this functor is a linear space over
$\Pic^c(X)$ whose underlying set is $\coprod_{l\in \Pic^c(X)}
H^1(\resto{\poincare}{\{l\}\times X})$, where   $\poincare$ is a Poincaré line
bundle normalized at $x_0$. 

The main idea idea of the proof is to compare the representability of $\Picaff_{X,x_0}^c$ to the representability of a functor considered by Bingener related to the deformation theory of $p$-cohomology classes. 

Our arguments show in particular that,  for $p=1$, the converse of  Bingener's representability criterion holds.     
\end{abstract}


\section{Introduction}

Let $X$ be compact complex space. The Picard $\Pic_X$ functor associated with $X$ is defined by $\Pic_X(T):=\Pic(T\times X/T)$, where

$$\Pic(T\times X/T):=\coh{0}{T}{R^1f_{T_*}(\shf{T\times X}^*)}\,. $$

A fundamental result of Grothendieck \cite[Théorème 3.3]{grothendieck1} states
that this functor is representable by a complex space endowed naturally with a
complex group structure.   

Bingener proved a relative version of this theorem: the relative Picard
functor $\Pic_{X/S}$ associated with a morphism $X\xra{f} S$ is representable
if $f$ is flat, proper and 
cohomologically flat in dimension $0$ \cite[Kapitel II.6]{bingener1}. He also
gave several general representability criteria for similar functors. 
In this paper we are interested in the classification of affine line bundles
over a compact complex space $X$, and the corresponding functor 
$\Picaff_X$. By a classical theorem \cite[p. 41, Theorem 3.2.1]{hirzebruch} 
the isomorphism classes of holomorphic affine line bundles on $X$ correspond
bijectively to the elements of the cohomology \emph{set} 
$\Picaff(X)=\coh{1}{X}{\Affshf(1)}$, where $\Affshf(1)$ denotes the sheaf of
(locally defined) holomorphic 
functions with values in the affine group $\gaff$. Since this group is
non-abelian, the classification problem for affine line bundles is more
difficult and interesting than the one for (linear) line bundles.
There is a subtle problem when one tries to find an alternative description for
$\Picaff(X)$. 
\begin{theorem}
  The isomorphism classes of affine line bundles on $X$ correspond bijectively to
  isomorphism classes of sheaf epimorphisms $p:\mathcal{E}\to {\cal O}_X$, where $\mathcal{E}$ is a locally free sheaf of rank 2.
  \end{theorem}
One can try to use this bijection to give a set-theoretical description of $\Picaff(X)$. 
 The kernel of such en epimorphism is a locally free sheaf of rank $1$ (i.e. a line
 bundle on $X$). Choosing a Poincaré line bundle $\poincare$ on $\Pic(X)\times
 X$ one obtains an identification (see section \ref{sec:groupoids}): 
\begin{equation}
\Picaff(X)=\coprod_{l\in \Pic(X)}
\coh{1}{X}{\resto{\poincare}{\{l\}\times X}}/\cstar\,.
\end{equation}
This quotient is non-Hausdorff  (unless $h^1$ is identically zero on $\Pic(X)$),  so it does have a complex analytic structure compatible with the quotient topology.

A natural idea to overcome this difficulty is to consider linearly \emph{framed} affine line
bundles in the following sense:

 Fix a  point $x_0\in X$, and define a sheaf $\Affshf(1)_{x_0}$ by
  $$\Affshf(1)_{x_0}(U)=\{f\in \Affshf(1)(U)|\ \lambda(f(x_0))=1 \text{ if } x_0\in
  U\}\,,$$
  where $\lambda:\gaff\to\C^*$ is the natural group morphism.
  Furthermore we define
  $$\Picaff(X)_{x_0}= \coh{1}{X}{\Affshf(1)_{x_0}}\,.$$
Moreover, for a class $c\in\NS(X)$ we denote by $\Picaff(X)_{x_0}^c$ the pre-image of $\Pic^c(X)$ under the natural map $\coh{1}{X}{\Affshf(1)_{x_0}}\to \coh{1}{X}{{\cal O}^*_X}$. 
 Then 
the following result hold.
\begin{theorem}
  $\Picaff(X)_{x_0}$ $(\Picaff(X)_{x_0}^c)$ is naturally identified to the set of isomorphism classes of pairs $(A,\epsilon)$,
  where $A$ is an affine line bundle on $X$ (of Chern class $c$) and $\epsilon:\C \to \lambda(A)_{x_0}$ is a
  linear isomorphism. Morphisms of such pairs are defined in the obvious way.
  Furthermore one has a natural set theoretical identifications:
  \begin{equation}
    \label{eq:master}
\begin{split}
  \Picaff(X)_{x_0}&=\coprod_{l\in \Pic(X)}
\coh{1}{X}{\resto{\poincare}{\{l\}\times X}}\,, \\
 \Picaff(X)_{x_0}^c&=\coprod_{l\in \Pic^c(X)}
\coh{1}{X}{\resto{\poincare}{\{l\}\times X}}\,.
\end{split}
  \end{equation}
 \end{theorem}
In this theorem, the Chern class of an affine line bundle is by definition the Chern class of its linearisation. 

This theorem shows that the set $\Picaff(X)_{x_0}^c$ is naturally
a union of vector spaces  parameterised  by $\Pic^c(X)$,
so one might think that it should come with a natural linear space  structure \cite[p. 49]{fischer}, and that this linear space classifies families of linearly framed affine line bundles (of fixed Chern class $c$) on $X$.
{\it We will prove that this is not always the case,
and we will give a necessary and sufficient condition for the existence of such a structure. }
%

The condition ``classifies families of linearly framed affine line bundles on $X$" mentioned above can be formalised as follows: We define the functors 
$$\Picaff_{X,x_0}:\cat{An}^{op} \ra
\cat{Set}\,,\ \Picaff_{X,x_0}^c:\cat{An}^{op} \ra
\cat{Set}$$
in the following way: for a complex space $T$ let $\Affshf(1)_{X_T,x_0}$
be the sheaf of germs of $\gaff$-valued morphisms $f:X_T\supseteq U \ra \gaff$
such that $\resto{\lambda \circ f}{U\cap (T\times \{x_0\})}\equiv 1$. 

Fixing a Chern class $c\in \NS(X)$ will yield a subsheaf (see section \ref{sec:relativecase} for details) 
$$R^1_c{f_T}_*\big(\Affshf(1)_{X_T,x_0}\big)\subset R^1{f_T}_*
(\Affshf(1)_{X_T,x_0} )\,.$$
 The functors  $\Picaff_{X,x_0}$, $\Picaff_{X,x_0}^c$ are defined by
$$T \mapsto \coh{0}{T}{R^1{f_T}_* (\Affshf(1)_{X_T,x_0} )},\  T \mapsto \coh{0}{T}{R^1_c{f_T}_* (\Affshf(1)_{X_T,x_0} )}$$
(see \cite[Chapitre V.§3]{giraud}  for the definition  of the functor $R^1$
for sheaves on non-Abelian groups).

The identification (\ref{eq:master})  suggests us to consider the  functor 
$$B^c_{X,x_0}:(\cat{An}/\Pic^c(X))^{op} \ra \cat{Set}$$
%
$$B^c_{X,x_0}(T\xra{g} \Pic^c(X))=\coh{0}{T}{R^1 {f_T}_*(g\times \id)^*\poincare^c}\,,$$
where $f$ is the projection $\Pic^c(X)\times X \ra \Pic^c(X)$.
Associating to $B^c_{X,x_0}$ the functor ${\widetilde B^c_{X,x_0}}$ as in Appendix \ref{functorstuff} 
 we formulate our main comparison theorem as follows:
 \begin{theorem}
  \label{thm:main} 
  The functor $\Picaff_{X,x_0}^c$ is representable by a complex space $V$ if and
  only if $\tilde{B^c_{X,x_0}}$ is represented by $V$. Moreover, if $V$ represents
  those functors it has a natural structure of a linear space over $\Pic^c(X)$.
 \end{theorem}
Clearly $B^c_{X,x_0}$ is a much more manageable functor, because $\poincare$ is a locally free sheaf
and $f$ is just a projection.  


Functors of this type have been studied before.
Proposition  \cite[Satz 8.1]{bingener1} gives a sufficient representability
condition. We will prove that, in our case, this condition is also necessary.  This will give us
the following
\begin{theorem}\label{thm:final}
The functor $\Picaff^c_{X,x_0}$ is representable  if
and only if the function $l \mapsto
  h^0(X_l,\resto{\poincare^c}{l})$ is constant on $\Pic^c(X)$. If this is the
  case, then it is represented by a linear fibre space over $\Pic^c(X)$. 
\end{theorem}

\section[The affine Picard functor]{The affine Picard functor. The linearly framed affine Picard functor}
\label{sec:defc}
For the rest of the article $X$ denotes a compact complex space. For two complex
spaces $X,T$, $X_T=T\times X$ and $f_T:X_T\ra T$ is defined as the first projection.

\subsection{The linear case and Grothendieck's representability theorem}

First we will review a classic result and present a new way to look at it with
the method developed in the Appendix. This
will give a clear motivation for the following chapters.
Define the functor  $\Pic_X: \An^{op}\to \Set$ by
$$\Pic_X(T):=\Pic(X_T/T)=H^0(T, R^1f_{T*}({\cal O}_{X_T}^*))\,.$$
Recall that the set $H^1(X,{\cal O}_X^*)$ has a natural structure of an abelian
complex Lie group, which will be denoted $\Pic(X)$. The connected component of
the 0-element is the quotient  $H^1(X,{\cal O})/H^1(X,\Z)$, which is Hausdorff,
because $H^1(X,\Z)$ is closed in $H^1(X,{\cal O})$ \cite[Proposition
3.2]{grothendieck1}.

\begin{theorem}\emph{\cite[Théorème 3.3]{grothendieck1}}
 The functor $\Pic_X$ is represented by $\Pic(X)$.
\end{theorem}
We refer to \cite[Kapitel II.6]{bingener1} for a generalization of this theorem to a
relative complex space $X\ra S$.
In the special case  $H^0({\cal O}_X)=\C$ we have an alternative interpretation of the functor $\Pic_X$:
\begin{remark}\label{Pic2functors}
Suppose $H^0({\cal O}_X)=\C$.  $\Pic(X)$ also represents the functor  
$$T\mapsto \Pic(X_T)/\Pic(T)\,. $$
\end{remark}

One obtains a new interpretation of the functor $\Pic_X$ using framed holomorphic line bundles. 
 Let $\Pic(T\times X, T\times \{x_0\})$ be the set of isomorphism classes of
 holomorphic line bundles on $T\times X$ framed along $T\times \{x_0\}$ . This
 set has a cohomological interpretation obtained as follows (see Appendix
 \ref{appendix} for details). Let ${\cal O}^*_{X_T,x_0}\subset {\cal O}^*_{X_T}$
 be the subsheaf defined by 
$${\cal O}^*_{X_T,x_0}:=\ker\big({\cal O}^*_{X_T}\to {\cal O}^*_{T\times\{x_0\}}\big)\,.
$$
Supposing $H^0({\cal O}_X)=\C$,  for  any complex space $T$  we have  canonical identifications
\begin{align*}
  \Pic(T\times X,T\times\{x_0\})&=H^1(T\times X,{\cal O}^*_{X_T,x_0})=H^0(T,R^1f_{T*}({\cal O}^*_{X_T,x_0}))\\
   &\simeq H^0(T,R^1f_{T*}({\cal O}^*_{X_T}))=\Pic_X(T)
\end{align*}

\begin{remark}\label{Pic2functorsnew}
Suppose $H^0({\cal O}_X)=\C$, and fix $x_0\in X$.  $\Pic(X)$ also represents the functor  
$$ T\mapsto \Pic(T\times X, T\times\{x_0\})=H^1(T\times X,{\cal O}^*_{X_T,x_0})=H^0(T,R^1f_{T*}({\cal O}^*_{X_T,x_0}))\,.$$
\end{remark}
A Poincaré line bundle $\poincare$ normalized at $x_0$ is a representative
of the element $\Pic(\Pic(X)\times X, \Pic(X)\times\{x_0\})$ which corresponds
to $\id_{\Pic(X)}$ under the isomorphism
$$\Pic(\Pic(X)\times X, \Pic(X)\times \{x_0\}) \ra \Hom(\Pic(X),\Pic(X))\,.$$
Such a Poincaré line bundle $\poincare$  defines an explicit isomorphism between functors 
$$T \mapsto \Hom(T, \Pic(X)),\ T\mapsto \Pic(T\times X, T\times\{x_0\})
$$
given  by
$$\varphi\mapsto (\varphi\times\id_X)^*(\poincare)\,.
$$

\subsection{The groupoid of affine line bundles and linearly framed  affine line bundles}
\label{sec:groupoids}

The set of isomorphism classes of holomorphic affine line bundles on $X$ is
denoted by $\Picaff(X)$. This pointed set (pointed via the trivial bundle) has a cohomological interpretation
given by the natural identification  \cite[I.Theorem 3.2.1]{hirzebruch} 

$$\Picaff(X)=\coh{1}{X}{\Affshf(1)_X}\,.$$

Let $A$ be an affine line bundle over $X$.
The frame bundle of $A$ is defined analogously to the frame bundle of a
line bundle.   More precisely, for  $x\in X$ put 
 $$F_x=\{p:\C\ra A_x|\ p \text{ is an affine linear isomorphism} \}\,,$$
and let $F(A)$ be the disjoint union $F(A)=\bigsqcup_{x\in X}F_x$. A point in
$F(A)$ is a pair $(x,p)$ where $x\in X$ and $p\in F_x$. One can use local
trivialisations of $A$ to define a complex manifold structure on $F(A)$ such
that, with respect to the obvious right $\gaff$-action, the projection  $F(A)\ra
X$ given by $(x,p)\mapsto x$ becomes a principal $\gaff$-bundle.
\begin{definition}
 The \emph{frame bundle} of $A$ is the map $F(A)\to X$ endowed with its natural
 principal $\gaff$-bundle structure over $X$.
\end{definition}
Denote by $\lambda:\gaff\ra \C^*$ be right hand morphism in the canonical short exact sequence
$$0\ra \C \ra \gaff \ra \C^* \ra 1\,. $$
\begin{definition}
  \label{def:linearisation}
  The \emph{linearisation} of an affine line bundle $A$ is    the
  associated bundle
  $$\lambda(A)= F(A)\times_\lambda \C\to X\,,$$
  where $\C^*=\GL(1,\C)$ acts on $\C$ in the natural way.
\end{definition}
Via linearisation, one obtains a map denoted by the same symbol,
$$\lambda:\Picaff(X)\ra \Pic(X)\,.$$
We have a group monomorphism $\varphi:\gaff \hookrightarrow \GL(2,\C)$ obtained by identifying $\gaff$
with the subgroup 
$$\left\{ \begin{pmatrix} a & b \\ 0 & 1 \end{pmatrix}  , a\in \C^*, b\in \C\right\}.$$
of $\GL(2,\C)$.

The groupoid $\cat{A}_X$ of affine line bundles on $X$ is defined as follows: ${\cal O}b(\cat{A}_X)$ is the class  of holomorphic affine
line bundles on $X$, and the morphisms in this category are affine line bundle isomorphisms over the identity. Furthermore we
define the groupoid $\cat{E}_X$ to be the category whose objects are bundle eimorphisms $ \mathcal{E}\textmap{p} {\cal O}_X$, where $\mathcal{E}$ is a locally free sheaf of rank 2. The morphisms  from $\mathcal{E}'\textmap{p'} {\cal O}_X$ to $\mathcal{E}''\textmap{p''} {\cal O}_X$ are  
isomorphisms  $\mathcal{E}'\textmap{u\simeq}\mathcal{E}''$ such that $p''\circ
u=p'$; i.e., such that the following diagram commutes.
\begin{equation*}
  \begin{tikzcd}
    \mathcal{E}'\ar[rr,"u"] \ar[rd,swap,"p'"] & & \mathcal{E}'' \ar[ld,"p''"] \\
      & \shf{X} &   
  \end{tikzcd}
\end{equation*}
We can think of $\cat{E}_X$ as the groupoid of vector bundle epimorphisms $E\ra X\times \C$ from a
rank 2 vector bundle $E$ to the trivial line bundle $X\times \C$ and we will use
this interpretation from now on.
For an affine line bundle $A$ let $F(A)\xra{q} X$ be the frame bundle defined above. 
We define a functor $\Phi:\cat{A}_X \ra \cat{E}_X$ by sending $A$ to $p:E\ra
X\times \C$ as follows:
\begin{itemize}
\item Define $E:=F(A)\times_\varphi \C^2$. 
\item The morphism $q\times \pr_2:F(A)\times \C^2 \ra X\times \C$ induces a vector bundle epimorphism
  $p:E\ra X\times \C$.
  
  The point here is that $\pr_2: \C^2\to\C$ is $\gaff$-invariant via $\varphi$.
\end{itemize}

We also have a functor $\Phi':\cat{E}_X\to \cat{A}_X$ which assigns to $p:E\ra X\times \C$ the pre-image of the section $1\in\Gamma(X\times \C)$.  
See \cite{plechinger} for the details of this construction and the proof of
the following theorem.
\begin{theorem}\label{Eq1}
  The pair   $(\Phi,\Phi')$ exhibits an equivalence of groupoids between $\cat{A}_X$ and $\cat{E}_X$.
\end{theorem}
We introduce a third groupoid $\cat{L}_X$ whose objects are pairs $(\lb,h)$, where
$\lb$ is a locally free sheaf of rank 1 over $X$ and $h\in \coh{1}{X}{\lb}$. For two
pairs $(\lb,h),(\lb',h')$ a morphism is an isomorphism $f:\lb\ra \lb'$ such that
$f^*(h')=h$.
Next we define a functor $\Psi:\cat{E}_X\ra \cat{L}_X$ by mapping an epimorphism
$\mathcal{E}\xra{q} \shf{X}$ to $(\lb,h)$, where
\begin{itemize}
\item $\lb$ is the kernel of $q$,
\item $h$ is the element of $\coh{1}{X}{\lb}=\Ext^1(\shf{X},\lb)$ defined by the   extension
  $$
  \begin{tikzcd}
0 \rar & \lb \rar[hook] & \mathcal{E} \rar["q"] & \shf{X} \rar & 0\, 
  \end{tikzcd}
  $$
  of $\shf{X}$ by $\lb$.
\end{itemize}

\begin{theorem}\label{Eq2}
  The functor $\Psi$ exhibits an equivalence of groupoids between $\cat{E}_X$ and $\cat{L}_X$.
\end{theorem}
Note that in this case we do not have an obvious functor $\Psi':\cat{L}_X\to
\cat{E}_X$ such that $(\Psi,\Psi')$ is an equivalence of categories. On the
other hand it is easy to check (using the theory of extensions) that $\Psi$ is
fully faithful and essentially surjective.

Theorems \ref{Eq1} and \ref{Eq2} show that the three groupoids introduced above
have the same isomorphism classes of objects. \\ 

Let $l\in\Pic(X)$ be an isomorphism class of line bundles on $X$. We define
$\mathfrak{H}^1(l)$ to be set of isomorphism classes of pairs $({\cal L},h)$ as
above with ${\cal L}\in l$. Note that two pairs $({\cal L},h)$, $({\cal L},h')$
are isomorphic if and only if there exists $\zeta\in\C^*$ such that $h'=\zeta
h$. This shows that, for any representative ${\cal L}\in l$, we have a {\it
  canonical identification} $\mathfrak{H}^1(l)\simeq H^1(X,{\cal L})/\C^*$.

With this notation the set isomorphism classes of objects of  the groupoid
$\cat{L}_X$ decomposes as the disjoint union  $\bigsqcup_{l\in
  \Pic(X)}\mathfrak{H}^1(l)$. Combining with Theorems \ref{Eq1}, \ref{Eq2} we
obtain our first classification theorem which gives a set theoretical
description of $\Picaff(X)$:  
\begin{theorem}\label{thm:SetTh}
  The set of isomorphism classes of holomorphic affine line bundles on $X$ is
  given by the disjoint union
  $$ \Picaff(X)=\bigsqcup_{l\in \Pic(X)} \mathfrak{H}^1(l)\,.$$
\end{theorem}
Therefore, for any line bundle ${\cal L}$ on $X$ the set $\Picaff(X)$ contains the  quotient 
$$\mathfrak{H}^1([{\cal L}])=H^1(X,{\cal L})/\C^*$$
 (which is non-Hausdorff when $H^1(X,{\cal L})\ne \{0\}$). This shows that, in
 general one cannot hope to have an analytic space parameterizing the
 isomorphism classes of affine line bundles.  On the other hand we will see that
 the classification problem for {\it linearly framed affine line bundles}, introduced
 below, is much more manageable. 

Fix a point $x_0\in X$. A linear $x_0$-frame (or, simply a frame, when $x_0$ is fixed)
for an affine line bundle $A$ is a linear isomorphism $\C\xra{\tau} \lambda(A)_{x_0}$.
In the language of Appendix \ref{appendix}, this is an $(x_0,\lambda)$-frame,
where $\lambda:\gaff \ra \C^*$ is the linearisation map used in \ref{def:linearisation}.
\begin{definition}
A linearly $x_0$-framed affine line bundle on $X$ is a pair $(A,\tau)$ consisting of a
holomorphic  affine 	line bundle on $X$ and a linear $x_0$-frame of $A$. 
\end{definition}
An isomorphism $g:(A',\tau')\ra (A'',\tau'')$ is an
isomorphism $ {g}:A'\ra A''$ such that $\lambda( {g})_{x_0}\circ \tau'=\tau''$.
\begin{definition}
The set of isomorphism classes of linearly framed affine line bundles is denoted by $\Picaff(X)_{x_0}$.
\end{definition}
Similarly to the non-framed framework we have a cohomological interpretation of
this set. Let $\Affshf_{X,x_0}$ be the sheaf of (locally
defined) functions $X\supset U\xra{f}\Aff(1,\C)$ such that $\lambda\circ f (x_0)= 1$ if
$x_0\in U$. We then get the identification:
$$\Picaff(X)_{x_0}=\coh{1}{X}{\Affshf(1)_{X,x_0}}\,.$$
We introduce analogues of the three groupoids above in
the linearly framed framework. Let $\cat{A}_{x_0}$ be the groupoid whose objects are linearly
$x_0$-framed affine line bundles on $X$, i.e. 
pairs $(A,\tau)$ consisting of an affine line bundle $A$ on $X$, and a frame
$\C\xra{\tau} \lambda(A)_{x_0}$ (see section \ref{appendix}).

Correspondingly, the objects of the groupoid $\cat{E}_{x_0}$ are pairs $({\cal
  E}\textmap{p}{\cal O}_X,\tau)$, where ${\cal E}\textmap{p}{\cal O}_X$ is an
object of $\cat{E}_X$, and $\C\xra{\tau} \ker(p)(x_0)$ is a linear
isomorphism. A morphism  
 $$({\cal E}\textmap{p}{\cal O}_X,\tau)\to ({\cal E}'\textmap{p'}{\cal O}_X,\tau')$$
   in $\cat{E}_{x_0}$ is a   morphism $f:({\cal E},p)\to ({\cal E}',p')$ in
   $\cat{E}_X$ such that  $\tau'=f_{x_0}\circ \tau$. 

Finally we define $\cat{L}_{x_0}$ to be the groupoid whose objects are triples
$(\lb,h,\tau)$ such that $(\lb,h)$ is an object in $\cat{L}_X$, and $\tau:\C\ra
\lb_{x_0}$ is a linear isomorphism.

A morphism $f:(\lb',h',\tau')\ra
(\lb'',h'',\tau'')$ in $\cat{L}_{x_0}$ is a morphism 
$$ {f}:(\lb',h')\ra (\lb'',h'')$$
 in $\cat{L}_X$ such that $\tau''= {f}_{x_0}\circ \tau'$.

One easily checks that as above we get functors (denoted by the same symbols)
$\Phi_{x_0}:\cat{A}_{x_0}\ra \cat{E}_{x_0}$ and 
$\Psi_{x_0}:\cat{E}_{x_0}\ra \cat{L}_{x_0}$ as well as analogues of the theorems
Theorem \ref{Eq1}, Theorem \ref{Eq2} above. 
\begin{theorem}\label{groupoids}
The functors $\Phi_{x_0}:\cat{A}_{x_0}\ra \cat{E}_{x_0}$ and $\Psi_{x_0}:\cat{E}_{x_0}\ra \cat{L}_{x_0}$ are equivalences of groupoids.
\end{theorem}
Let $\poincare \ra \Pic(X)\times X$ be a Poincaré line
bundle normalised at $x_0$. For a point $l\in\Pic(X)$ we denote by
$\resto{\poincare}{l}$ the line bundle on $X$ given by the restriction of
$\poincare$ on $\{l\}\times X$. 

By definition $\poincare$ comes with a trivialisation on
$\Pic(X)\times\{x_0\}$, so with a holomorphic family of linear isomorphism  
$$\tau_l:\C\to \resto{\poincare}{l}(x_0),\ l\in\Pic(X)\,.$$
For any triple $(\lb,h,\tau)\in{\cal O}b(\cat{L}_{x_0})$ we get {\it a well defined} isomorphism 
$$\lb\xra{f_\tau}\resto{\poincare}{[{\lb}]}$$
satisfying the condition $f_\tau(x_0)\circ \tau=\tau_{[{\lb}]}$. Therefore we
get a well defined cohomology class
$f_\tau(h)\in H^1(X,\resto{\poincare}{[{\cal L}]})$.
For two isomorphic triples
$({\cal L},h,\tau)\simeq ({\cal L}',h',\tau')$ one has
$f_\tau(h)=f_{\tau'}(h')$. Moreover, for any fixed 
isomorphism class $l\in\Pic(X)$, the map 
$${\cal O}b(\cat{L}_{x_0})\ni ({\lb},h,\tau)\mapsto f_\tau(h)\in H^1(X,\resto{\poincare}{l})\hbox{ (where $\lb$ is a representative of $l$)}$$
 identifies the set of isomorphism classes of objects $({\cal L},h,\tau)$ with
 $[{\cal L}]=l$ with $H^1(X,\mathfrak{L}_l)$. This gives the following
 description of the set of isomorphism classes $\Picaff(X)_{x_0}$. 
\begin{theorem}
\label{thm:picaff}
 Let  $\poincare\ra \Pic(X)\times X$ be a Poincaré line bundle normalised at $x_0$. One has a natural identification 
   $$\Picaff(X)_{x_0}=\bigsqcup_{l\in \Pic(X)}\coh{1}{X}{\resto{\mathfrak{L}_{x_0}}{\{l\}\times X}}\,.$$
\end{theorem}
Therefore the set of isomorphism classes of linearly $x_0$-framed affine line bundles
decomposes as a disjoint union of finite dimensional complex vector spaces.

\subsection{The relative case} 
\label{sec:relativecase}
Theorems \ref{groupoids} and \ref{thm:picaff} have relative versions which will play an important role in our proofs. For a complex space $T$, we define the groupoids $\cat{A}_{X,x_0}(T)$,  $\cat{E}_{X,x_0}(T)$, $\cat{L}_{X,x_0}(T)$ as follows:
\begin{enumerate}
\item The objects of $\cat{A}_{X,x_0}(T)$ are linearly $(T\times\{x_0\},l)$-framed affine line bundles $(A,\tau)$ on $T\times X$ (see Appendix \ref{appendix}).
\item The objects of  $\cat{E}_{X,x_0}(T)$ are pairs
$$(p:{\cal E}\to {\cal O}_{T\times X}, \tau: {(\ker p)}_{T\times\{x_0\}}\xra{\simeq} {\cal O}_{T\times\{x_0\}})\,,$$
where ${\cal E}$ is a locally free sheaf of rank 2 on $T\times X$, and $p$ is an epimorphism.
\item The objects of  $\cat{L}_{X,x_0}(T)$ are triples  $({\cal L},h,\tau)$, where
$({\cal L},\tau)$ is a $(T\times\{x_0\})$-framed line bundle on $T\times X$ and
$h\in H^1(T\times X, {\cal L})$.
	
\end{enumerate}

In each case the morphisms in the three groupoids are isomorphisms of pairs,
respectively triples, in the obvious sense.

For a complex space $T$ define $\Affshf(1)_{X_T,x_0}$
as the sheaf of germs of $\gaff$-valued morphisms $f:X_T\supseteq U \ra \gaff$
such that $\resto{\lambda \circ f}{U\cap T\times \{x_0\}}\equiv 1$. 

\begin{remark}\label{RemRel1}
The set of isomorphism classes of linearly $(T\times\{x_0\},l)$-framed affine line
bundles $(A,\tau)$ on $T\times X$ is identified with $H^1(T\times X,\Affshf(1)_{X_T,x_0})$.
\end{remark}

Analogously to the non-relative version we obtain functors 
$$\Phi_{x_0}(T):\cat{A}_{X,x_0}(T)\ra \cat{E}_{X,x_0}(T),$$
$$\Psi_{x_0}(T):\cat{E}_{X,x_0}(T)\ra \cat{L}_{X,x_0}(T)$$

\begin{theorem}
  The functors $\Phi_{x_0}(T)$ and $\Phi_{x_0}(T)$ are equivalences of groupoids.
\end{theorem}

Let again $\poincare$ be a Poincaré line bundle normalized at $x_0$. This means that we have a fixed a framing $\mathfrak{t}_{x_0}:\resto{\poincare}{\Pic(X)\times\{x_0\}}\xra{\simeq} {\cal O}_{\Pic(X)\times\{x_0\}}$.
\begin{remark}\label{RemRel2}
Isomorphism classes of triples $({\cal L}, h,\tau)\in {\cal O}b(\cat{A}_{X,x_0}(T))$   correspond bijectively to pairs $(\varphi,\sigma)$ where 
$$\varphi:T\to\Pic(X),\ \sigma\in H^1((\varphi\times\id_X)^*(\poincare))\,.$$
\end{remark}
\begin{proof}
Let $({\cal L}, h,\tau)\in {\cal O}b(\cat{A}_{X,x_0}(T))$, and let $\varphi:
T\to \Pic(X)$ the map defined by ${\cal L}$, so ${\cal L}\simeq
(\varphi\times\id_X)^*(\poincare)$. Since $X$ is compact and $H^0(X)=\C$, the
set of isomorphisms ${\cal L}\xra{\simeq} (\varphi\times\id_X)^*(\poincare)$ is
an $H^0(T,{\cal O}_T^*)$-torsor, so   there exists a unique such isomorphism
which maps $\tau$ on the pullback framing
$(\varphi\times\id_X)^*(\mathfrak{t}_{x_0})$ of
$(\varphi\times\id_X)^*(\poincare)$. Using this isomorphism we obtain a well
defined class $\sigma\in H^1((\varphi\times\id_X)^*(\poincare))$.

 \end{proof}

Using Remarks \ref{RemRel1},  \ref{RemRel2} we obtain  canonical identifications
\begin{equation}\label{ident} 
\hspace*{-2mm}H^1(X_T,\Affshf(1)_{X_T,x_0}) \xra{\simeq} \{(\varphi,\sigma)|\  T\xra{\varphi} \Pic(X),\ \sigma \in H^1(X_T,(\varphi\times\id_X)^*\poincare)\} 	
\end{equation}
which are functorial with respect to $T$.
Now fix $T$, and use the identifications (\ref{ident}) associated with open sets $U\subset T$. We obtain a presheaf isomorphism 
\begin{equation}\label{F} 
F^T_{x_0}:{\cal A}_{x_0}^T\to {\cal C}^T_{x_0}\,,
\end{equation}
where ${\cal A}_{x_0}^T$, ${\cal C}_{x_0}^T$ are the presheaves of sets on  $T$ defined by
$$U\mapsto H^1(\Affshf(1)_{X_U,x_0}),\ U\mapsto \{(\varphi,\sigma)|\
\varphi:U\to \Pic(X),\ \sigma \in H^1(\varphi\times\id_X)^*(\poincare))\}$$  
respectively.  We are interested in the sheaves $\widetilde{{\cal A}_{x_0}^T}$,
$\widetilde{{\cal C}_{x_0}^T}$ associated with these presheaves. 

\begin{lemma} The presheaf ${\cal B}_{x_0}^T$ of sets on $T$ given by 
$$U\mapsto \{(\varphi,\sigma)|\ \varphi:U\to \Pic(X),\ \sigma \in H^0(R^1 f_{U*}
(\varphi\times\id_X)^*(\poincare))\}\,.$$ 
is a sheaf. There exists a  natural presheaf morphism $\eta^T_{x_0}:{\cal
  C}^T_{x_0}\to {\cal B}_{x_0}^T$ which induces an isomorphism
$\widetilde{\eta^T_{x_0}}:\widetilde{{\cal C}_{x_0}^T}\to {\cal B}_{x_0}^T$. 
\end{lemma} 
\begin{proof}  The first claim is obvious. The natural  presheaf morphism
  $\eta^T_{x_0}$ is given, for any open set $U\subset T$  by
  $\eta^U_{x_0}(\varphi,\sigma):=(\varphi,\tilde \sigma)$, where $\tilde \sigma$
  is the image of $\sigma$ under the canonical map 
$$\coh{1}{X_U}{(\varphi \times \id_X)^*\poincare}\to
\coh{0}{U}{R^1{f_U}_*(\varphi \times \id_X)^*\poincare}\,.$$
For the last statement it suffices to show that the topology of $T$ has a basis
${\cal U}$ such that $\eta^U_{x_0}$ is a bijection for any $U\in {\cal U}$.
Using Cartan's Theorem B and the Leray spectral sequence associated with the
sheaf $(\varphi \times \id_X)^*\poincare$ and the proper projection $U\times
X\to U$, we see that $\eta^U_{x_0}$ is a bijection if $U$ is Stein. 

\end{proof}

The sheaf  $\widetilde{{\cal A}_{x_0}^T}$ associated with ${\cal A}_{x_0}^T$
is the non-abelian first direct image $R^1 {f_T}_{*}(\Affshf(1)_{X_T,x_0})$.  
\begin{remark}\label{rem:sheafiso}
Via the sheaf isomorphism  $\widetilde{\eta^T_{x_0}}$,  	the isomorphism $F^T_{x_0}$ induces a sheaf isomorphism
$$\mathbb{F}^T_{x_0}:R^1 {f_T}_{*}(\Affshf(1)_{X_T,x_0})\xra{\simeq} {\cal B}_{x_0}^T\,.
$$
\end{remark}

We are almost ready to proof our first main theorem. 

\begin{definition}
  The functor
  $$\Picaff_{X,x_0}:\An^{op} \ra \Set,$$
  defined by
  $$T \mapsto \coh{0}{T}{R^1{f_T}_*\Affshf(1)_{X_T,{x_0}}}$$
  is called the $x_0$-framed affine Picard functor.
\end{definition}

Furthermore define a functor ${\cal B}_{x_0}:\An^{op} \ra \Set$
which acts on objects by 
\begin{align*}
  {\cal B}_{X,x_0}(T)&:=H^0(T,{\cal B}_{X,x_0}^T) \\
                   &=\big\{(T\textmap{\varphi}\Pic(X),\sigma)|\ \sigma\in H^0(T,f_{T*}((\varphi\times\id)^*(\poincare))\big\}\,,
\end{align*}
and on morphisms $u:S\to T$ via the base change maps. 

Consider now  the functor 
$$B_{X,x_0}: (\An/\Pic(X))^{op} \to \Set$$
 associated with the sheaf $\poincare$ on $\Pic(X)\times X$, where $\Pic(X)\times X$ is regarded as a complex space over $\Pic(X)$:
$$B_{X,x_0}(T\xra{g} \Pic(X))=\coh{0}{T}{R^1 {f_T}_*(g\times
  \id)^*\poincare}\,.$$
Our functor ${\cal B}_{X,x_0}$ is related to the Bingener functor by
\begin{remark}\label{Bfunctors}
With the notation introduced in Appendix \ref{functorstuff} we have
$$\widetilde{B_{X,x_0}}={\cal B}_{X,x_0}\,.$$	
\end{remark}

Our first main theorem is the following:
\begin{theorem}
  \label{thm:main1} Let $X$ be a compact complex space with $H^0({\cal O}_X)\simeq\C$. 
  \begin{enumerate} 
  \item 	There exists an isomorphism of functors $\mathbb{F}_{x_0}:\Picaff_{X,x_0} \xra{\simeq} {\cal B}_{X,x_0}$.
   \item  The functor $\Picaff_{X,x_0}$ is represented by a complex space $V$ if and
  only if there exists $u:V\to \Pic(X)$  which represents  $B_{X,x_0}$. 
  \end{enumerate}
\end{theorem}

\begin{proof}  
  Recall that the identifications (\ref{ident}) are functorial in $T$. Then so
  is the sheaf isomorphism $\mathbb{F}_{x_0}^T$ from Remark \ref{rem:sheafiso}, and these isomorphisms  yield  an isomorphism of functors $\mathbb{F}_{x_0}:\Picaff_{X,x_0} \xra{\simeq} {\cal
    B}_{X,x_0}$. The claim follows now from Remark \ref{Bfunctors} and Proposition \ref{app:rep}.
\end{proof}

\begin{remark}
The projection on the first factor gives a morphism of functors 	$P_1:{\cal B}_{X,x_0}\to \mathrm{Hom}_{\An}(\, \cdot\,, \Pic(X))$. The composition $P_1\circ \mathbb{F}_{x_0}$ has an obvious geometric interpretation: it maps  a section 
$$\mathfrak{s}\in \coh{0}{T}{R^1{f_T}_*\Affshf(1)_{X_T,{x_0}}}= \Picaff_{X,x_0}(T)$$

 to the map $T\to \Pic(X)$ associated with the linearization 
 $$\sigma:=l_{x_0}(\mathfrak{s})\in  \coh{0}{T}{R^1{f_T}_*{\cal O}^*_{X_T,{x_0}}}$$
 (see Remark  \ref{Pic2functorsnew}).
\end{remark}

Fix now a Chern class $c\in \NS(X)$. We define the subsheaf
$$R^1_c\,{f_T}_*\Affshf(1)_{X_T,{x_0}}\subseteq R^1{f_T}_*\Affshf(1)_{X_T,{x_0}}\,,$$
as the sheaf associated to the presheaf $U\mapsto
H^1(X_U,\Affshf(1)_{X_U,x_0})^c$, where
$$H^1(X_U,\Affshf(1)_{X_U,x_0})^c\subseteq H^1(X_U,\Affshf_{X_U,x_0})$$
consists of the subset of isomorphism classes of
affine line bundles whose restriction to $X$ have Chern class $c$ (i.e. the
preimage of $\Pic^c(X)$ under the natural map $H^1(X_U,\Affshf_{X_U,x_0})\ra \Pic(X)$).
The functors
$$\Picaff^c_{X,x_0}:\An^{op}\ra \Set\,,
\mathcal{B}^c_{X,x_0}:\An^{op}\ra \Set\,,$$
are then defined via  
$$T\mapsto H^0(T,R^1_c\, {f_T}_*\Affshf(1)_{X_T,{x_0}})\,,$$
and
$$T\mapsto \{(T\xra{\varphi} \Pic^c(X),\sigma)|\ \sigma\in
H^0(T,f_{T_*}((\varphi\times \id)^*(\poincare^c)))\}\,. $$

One obtains as in \ref{thm:main1}:
\begin{corollary}
  \label{cor:maincor}
  Let $X$ be a compact complex space with $H^0({\cal O}_X)\simeq\C$. 
  \begin{enumerate} 
  \item There exists an isomorphism of functors $f_{x_0}:\Picaff^c_{X,x_0} \xra{\simeq} {\cal B}^c_{X,x_0}$.
  \item The functor $\Picaff^c_{X,x_0}$ is represented by a complex space $V$ if and
    only if there exists $u:V\to \Pic^c(X)$  which represents  $B^c_{X,x_0}$. 
  \end{enumerate}
\end{corollary}


\section{Representability}
\label{sec:representability}

Next we will look for conditions under which the above functors are
representable.  
Let $M\xra{f} S$ be a separated morphism,  $\mathcal{F}$ be an $S$-flat coherent
$\shf{M}$-module, whose support is proper over $S$, and $p$ be a natural number.
For a complex space $T\xra{g} S$ define
$$F_p^{\mathcal{F}}(T)=\coh{0}{T}{R^p {f_{T}}_*(\mathcal{F}_T)}\,,$$
where $M_T=T\times_S M, f_T=T\times_S f$ and $\mathcal{F}_T$ is the preimage of
$\mathcal{F}$ under the projection $M_T \ra M$.
Then $F_p^{\mathcal{F}}$ defines in a natural way a functor $(\An/S)^{op} \ra \Vect(\C)$, where $\Vect(\C)$
is the category of vector spaces over $\C$. We begin by recalling a representability theorem  of Bingener which will play an important role of our arguments:

\begin{theorem}[\cite{bingener1} Satz (8.1)(3)]
 \label{thm:representability-1} 
  With the notations and under the assumptions above the following holds:
  
  If $R^{p-1}f_*(\mathcal{F})_s \ra H^{p-1}(M_s,\mathcal{F}_s)$ is surjective
  for every point $s\in S$, then $F$ is represented by a linear fiber space over $S$.
\end{theorem}
We refer to \cite[p. 49]{fischer} for the theory  of linear  spaces in the complex analytic category.\\

Note that Bingener's result gives only a sufficient representability condition.
We will see that, in our special case, the surjectivity condition in Bingener's
Theorem \ref{thm:representability-1} is equivalent to the very simple condition
required in our Theorem \ref{thm:final}, and that this condition is also
necessary.

In our case, $M=\Pic(X)^c\times X$, $S=\Pic(X)^c$, $f=\pr_1$ is the first
projection and $\mathcal{F}=\poincare^c$.
Since $X$ is compact, all the conditions listed above are fulfilled. Moreover
the functor $B^c_{X,x_0}$ introduced in section \ref{sec:relativecase} coincides
with Bingener's functor $F^{\poincare^c}_{1}$, so Theorem
\ref{thm:representability-1} applies.  
Furthermore, we have $p=1$ so we need to examine the morphisms
\begin{equation}\label{eq:coflat}
f_*(\mathcal{F})_s\ra \coh{0}{M_s}{\mathcal{F}_s}.
\end{equation}

Recall the following definition from \cite[p. 116, p. 122]{banica}:
\begin{definition}
  Let $M\xra{f} S$ be morphism of complex spaces and  $\ff$ a coherent
  sheaf on $M$ with proper support, which is flat over $S$, and let $s\in S$.
  $\ff$ is
  \emph{cohomologically flat in dimension $p$} at $s$ if
the canonical maps
$$ R^{p-1}f_*(\ff)_s \ra H^{p-1}(X_s,\ff_s),\ 
R^{p}f_*(\ff)_s \ra H^{p}(X_s,\ff_s)\,, $$
are surjective. $\ff$ is called cohomologically flat in dimension $p$ over $S$ if
  it is cohomologically flat at every $s\in S$.
\end{definition}

\begin{remark}
  \label{cohplat0}
 Note that cohomological flatness in dimension 0 is equivalent to the surjectivity
of the maps $f_*(\ff)_s \ra H^0(X_s,\ff_s),$ for $s\in S$.
\end{remark}
A general criterion for cohomological flatness is given by the following
theorem, due to Grauert, see \cite[p. 134, Theorem 4.2 (ii)]{banica}.

\begin{theorem}[Continuity theorem]
  \label{continuity}
  Let $g:M\ra S$ be a proper morphism, $\mathcal{G}$ a coherent sheaf on $M$,
  flat over $S$ and $n$ a natural number.
  If $\mathcal{G}$ is cohomologically flat in dimension $n$ over $S$, then the
  function
  $$z \mapsto \dim \coh{n}{W_z}{\mathcal{G}_z}$$
  is locally constant. Conversely, if this function is locally constant and $S$
  is a reduced space, then $\mathcal{G}$ is cohomologically flat in dimension $n$
  over $S$. In particular, the sheaf $R^n g_*(\mathcal{G})$ is locally free.
\end{theorem}

Combining Theorem \ref{thm:representability-1}, Remark \ref{cohplat0} and Theorem
\ref{continuity}, we obtain:
\begin{corollary}
  \label{thm:locallyconstant}
  If the map $\Pic^c(X) \ni l \mapsto h^0(X_l,\resto{\poincare^c}{l})$ is constant,
  then the functor $\Picaff_{X,x_0}^c$ is represented by a linear space over
  $\Pic^c(X)$.
\end{corollary}

Before we continue discussing representability, let us look at a simple example,
a Riemann surface of genus 1. 
\begin{example}
  Let $X$ be a compact Riemann surface of genus $1$ and $c\neq 0$. The map
  $$\Pic^c(X)\ra \Z,\ l\mapsto \dim \coh{0}{X}{\resto{\poincare^c}{l}}$$
  is constant.\\
This follows from the fact that $h^{0}(\resto{\poincare}{l})=0$ if $c<0$ and
$h^1(\resto{\poincare}{l})=0$ if $c>0$ as well as Riemann-Roch
$$h^0(\resto{\poincare}{l}) = h^{1}(\resto{\poincare}{l}) +
\deg\resto{\poincare}{l} - g +1 = h^{1}(\resto{\poincare}{l}) +
\deg\resto{\poincare}{l}$$

On the other hand, if $c=0$, $l\ra \dim\coh{0}{X}{\resto{\poincare}{l}}$ is not
constant since $h^1(\resto{\poincare}{l})=1$ if $\resto{\poincare}{l}\iso \shf{X}$
and $h^1(\resto{\poincare}{l})=0$ otherwise. 
\end{example}

This shows that Theorem
\ref{thm:locallyconstant} does not apply in this case.
On the other hand, our Theorem \ref{thm:main-theorem}, proved later in this
section, shows that the  
functor $B^0_{X,x_0}$ cannot be represented by any complex space over $\Pic^0(X)$.

We are ready to prove the converse of Theorem \ref{thm:representability-1} in
the case that $S$ is reduced. That means we want to prove
the last part of the following theorem.

\begin{theorem}
 \label{thm:proof-part-2} 
  Let $f:X\ra S$ be a separated holomorphic map and $\mathcal{F}$ an $S$-flat
  coherent sheaf on $X$ whose support is proper over $S$.
  If $\mathcal{F}$ is cohomologically flat in dimension $0$ then the
   functor $F: (\An/S)^{op}\ra \Set$ given by
  $$F(T\ra S)=\coh{0}{T}{R^1{f_T}_*\mathcal{F}_T}$$
   is representable. Conversely, if $F$
    is representable and $S$ is reduced, then $\mathcal{F}$ is cohomologically flat in
    dimension $0$.
\end{theorem}

\begin{proof}

The implication ``$\mathcal{F}$ is cohomologically flat in dimension $0$" $\Rightarrow$  $F$ is representable is a special case of Bingener's theorem Theorem \ref{thm:representability-1} taking into account Remark \ref{cohplat0}.

We prove now that (assuming $S$ reduced), if $F$ is representable, then $\mathcal{F}$ is cohomologically flat in dimension $0$.

 Since $S$ is reduced, it suffices to prove
that the map $s\mapsto h^0(X_s,{\cal F}_s)$ is locally constant on $S$ by the
Continuity Theorem \ref{continuity}. We will prove that this map is constant on any irreducible component $S_0$ of $S$. Put
$$k_0:=\min\{h^0(X_s,{\cal F}_s)|\ s\in S_0\}\,,
$$
and let   $U_0\subseteq S_0$ be the
subset where this  minimal value is obtained. By Grauert's semi-continuity theorem, its complement
$Z_0:=S_0\setminus U_0$ is an analytic set of $S_0$. We will prove by reductio ad absurdum that $Z_0$ is empty.
Suppose this is not
the case, and let $z_0\in Z_0$. 
\vspace{1mm}\\

Denoting by $D$ the standard disk in $\C$ we note that 
\begin{remark}
  Let $Y$ be an irreducible complex space,  $Z\subsetneq Y$ be an analytic
  subset, and $z_0\in Z$. There exists a morphism  $h:D\to Y$ such that
  $h(0)=z_0$ and $h(D\setminus\{0\})\subset Y\setminus Z$. 
\end{remark}
\begin{proof} Let $n$ be dimension of $Y$ and ${\cal I}\subset {\cal O}_Y$ be
  the ideal sheaf of $Z$. Since the germs $(Z,z_0)$, $(Y,z_0)$ are different
  there exists $\psi\in {\cal I}_{z_0}\setminus \{0\}$. Note that, since the
  local ring ${\cal O}_{z_0}$ is an integral domain, $\psi$ is not zero divisor
  in this ring, so $\dim({\cal O}_{Y,z_0}/(\psi))=n-1$ by \cite[Corollary
  11.8]{atiyah}. Let $(H,z_0)$ be the germ defined by the principal ideal
  $(\psi)$, and note the inclusion of germs $(Z,z_0)\subset (H,z_0)$. 

 Let $\mathfrak{m}_0\subset {\cal O}_{z_0}$, $\bar{\mathfrak{m}}_0\subset  {\cal O}_{Y,z_0}/(\psi)$ be the maximal ideals of the respective local rings. By \cite[Theorem 11.14]{atiyah} it follows that there exist $\varphi_1, \dots,\varphi_{n-1}\in \bar{\mathfrak{m}}_0$ such that 
$$\sqrt{(\varphi_1, \dots,\varphi_{n-1})}=\bar{\mathfrak{m}}_0\,.$$
Let $\psi_i$ be a lift of $\varphi_i$ via the epimorphism $\mathfrak{m}_0\to
\bar{\mathfrak{m}}_0$, and note that the dimension of the local ring ${\cal
  O}_{Y,z_0}/(\psi_1,\dots,\psi_{n-1})$ is 1. Therefore the dimension of   the
germ  $(\Gamma,z_0)$ defined by the ideal $(\psi_1,\dots,\psi_{n-1})$ is 1. The
intersection of the germs $(H,z_0)$, $(\Gamma,z_0)$ reduces to $\{z_0\}$. Let
$g: (\tilde\Gamma,x_0)\to (\Gamma,z_0)$ be a normalisation of  the reduction of
$(\Gamma,z_0)$. It suffices to put $h:=g\circ u$, where $u:(D,0)\to
(\tilde\Gamma,x_0)$ is a biholomorphic parametrisation of a sufficiently small
neighborhood of $x_0$. 
\end{proof}

Let $h:D\ra S$ be a morphism as in the remark above.  Since $z_0\in Z_0$ we have 
\begin{equation}\label{larger}
h^0(X_{z_0},{\cal F}_{z_0})>k_0\,.	
\end{equation}
Regarding $D$ as a complex space over $S$ via $h$, let  $\mathcal{G}:=\ff_D$. Denote by ${\cal I}_{0}$   the
  ideal sheaf of the singleton $\{0\}\subset D$.   Since ${\cal F}$ is flat over $S$, it follows that ${\cal G}$ is flat over $D$, so the  exact sequence
  $$0 \ra \mathfrak{\cal I}_{0} \ra \shf{D} \ra \shf{\{0\}} \ra 0$$
  on $D$ gives an exact sequence of coherent sheaves on $X_D$:
  $$0 \ra \mathcal{G} \otimes  f_D^*({\cal I}_{0}) \ra \mathcal{G}  
  \ra \mathcal{G} \otimes  f_D^*(\shf{\{0\}})\ra 0\,.$$
  The associated pushforward long exact sequence on $D$ reads:
  \begin{center}
    
    \begin{tikzcd}
      0 \rar & f_{D_*}(\mathcal{G} \otimes  f_D^*({\cal I}_{0})) \rar &
      f_{D_*}(\mathcal{G})  \rar & f_{D_*}(\mathcal{G} \otimes
      f_D^*(\shf{\{0\}})) \\
      \rar &   R^1f_{D_*}(\mathcal{G} \otimes
      f_D^*({\cal I}_{0})) \rar{\phi} &  R^1f_{D_*}(\mathcal{G}) 
    \end{tikzcd}
  \end{center}
%
%
Denoting by $\zeta\in {\cal O}(D)$ the holomorphic map given by the embedding $D\hookrightarrow\C$, we see that ${\cal I}_{0}=\zeta{\cal O}_D$, in particular ${\cal I}_{0}$ is free of rank 1, and multiplication with $\zeta$ defines a sheaf isomorphism  ${\cal O}_D\textmap{\simeq} {\cal I}_{0}$, hence a sheaf isomorphism
$$R^1f_{D_*}(\mathcal{G})=R^1f_{D_*}(\mathcal{G}\otimes f_D^*({\cal O}_D))\textmap{\tilde \zeta\simeq } R^1f_{D_*}(\mathcal{G} \otimes  f_D^*({\cal I}_{0}))\,.
$$ 
The composition $\phi\circ\tilde\eta:R^1f_{D_*}(\mathcal{G})\to R^1f_{D_*}(\mathcal{G})$  is given by multiplication with $\zeta$ on this ${\cal O}_D$-module. 
%
%
By Proposition ~\ref{prop:torsion-free} proved below we know that $R^1f_{D_*}(\mathcal{G})$ is torsion free on $D$, so this sheaf morphism is injective. Since $\tilde\zeta$ is an isomorphism, it follows that $\phi$ is injective.

This implies that the canonical morphism $f_{D_*}(\mathcal{G})\ra
f_{D_*}(\mathcal{G}\otimes f_D^*(\shf{\{0\}})$ 
is an epimorphism, in particular the canonical morphism of ${\cal O}_{D,0}$-modules
$$f_{D_*}(\mathcal{G})_0\to f_{D_*}(\mathcal{G}/\mathfrak{m}_{0}\mathcal{G})_{0}$$
is surjective. By \cite[III. Theorem 3.4, Corollary 3.7]{banica} this is equivalent to $f$
being cohomologically flat in dimension $0$ at $0$. On the other hand the map $z\mapsto h^0({\cal G}_z)$ is constant $k$ on $D\setminus\{0\}$, so by Grauert's continuity theorem \cite[III. Theorem 4;12 (ii)]{banica} it follows that ${\cal G}$ is cohomologically flat in dimension 0 on $D\setminus\{0\}$. Therefore ${\cal G}$ is cohomologically flat on $D$, so again by  Grauert's continuity theorem, it follows that map  
 $$D\ni z\mapsto h^0((X_D)_z,{\cal G}_z)=h^0(X_{f(z)},{\cal G}_{f(z)})
 $$
 is constant on $D$, which contradicts (\ref{larger}).
 \end{proof}

Finally we apply the developped methods to the case of affine line bundles.
\begin{theorem}
  \label{thm:main-theorem}
  The functor $\Picaff_{X,x_0}^c$ is representable by a linear fiber space over $\Pic^c(X)$ if and only if the function $l \ra
  h^0(X_l,\resto{\poincare^c}{l})$ is locally constant.
\end{theorem}

We still need to add the following Proposition:
\begin{proposition}
  \label{prop:torsion-free}
  Let $f:X\ra S$ be a separated holomorphic map and $\mathcal{F}$ an $S$-flat
  coherent sheaf on $X$ whose support is proper over $S$.
  Suppose that the functor $F: (\An/S)^{op}\to \Set$ 
  $$F(T\ra S)=\coh{0}{T}{R^1{f_T}_*\mathcal{F}_T}\,,$$
   is representable. Let $g:T\ra S$ with $T$ locally irreducible. Then the
   $\shf{T}$-sheaf $R^1{f_T}_*\mathcal{F}_{X_T}$ is torsion-free.
 \end{proposition}
 %

\begin{proof}
  Fix a representation $\varphi:F \xra{\simeq} h_Y$ with $Y\xra{h} S$.
  For $g:T\ra S$, let $0_g\in F(g)$ be the trivial element and let
  $O_g=\varphi(0_g)\in \Hom_S(T,Y)$. We call this the zero section over $g$.
  Clearly, if $O_{\id}:=O_{\id_S}:S\ra Y$ is the zero section over $\id_S$, for
  every $g:T\ra S$ we have $O_g=g^* O_\id$.
  Suppose now that $T$ is irreducible and $\sigma \in \coh{0}{T}{\ff_g}$.
  Suppose $U\subseteq T$ is an open subset such that $\resto{\sigma}{U}=0$.
  Then:
  $$\resto{\sigma}{U}=0 \Leftrightarrow
  \resto{\varphi(\sigma)}{U}=\resto{O_g}{U}\,.$$
  Using the identity theorem, we can deduce that $\varphi(\sigma)=O_g$ and
  therefore $\sigma=0_g$.
\end{proof}


\section{Appendix}
\subsection[Appendix: Framing  a principal bundle]{Appendix 1}\label{appendix}
Some technical results needed in the later chapters are presented in this
section. Let $G$ be a complex Lie group,  and $X$ be a complex space. We denote by ${\cal G}$ the sheaf of  $G$-valued holomorphic morphisms defined on open sets of $X$. The cohomology set $H^1(X,{\cal G})$ can be identified with the set of isomorphism classes of $G$-principal bundles on $X$. Let now $\chi:G\to \C^*$ be a character,  and $Z\subset X$ be closed complex subspace. 
\begin{definition}
  Let ${\cal P}$ be a  principal $G$-bundle on $X$. 
  A $(Z,\chi)$-\emph{frame} of ${\cal P}$   is a trivialisation
  $\tau:\resto{{\cal P}\times_\chi \C}{Z} \xra{\simeq} Z\times \C$. A $(Z,\chi)$-framed $G$-bundle on $X$ is a  principal $G$-bundle endowed with a $(Z,\chi)$-\emph{frame}.
\end{definition}

Let $\lb^{\cal P}$ be the sheaf of sections of ${\cal P}\times_\chi \C$, which is given by
$$\lb^{\cal P}(U)=\big\{\varphi\in {\cal O}({\cal P}_U)|\ \varphi\circ R_g=\chi(g)^{-1}\varphi\ \forall g\in G\}\,.
$$
Here we denoted by $R_g$ the right translation associated with $g$. With these notations we have
\begin{remark}
A $(Z,\chi)$-frame of ${\cal P}$ is equivalent to the data of a sheaf isomorphism
$\resto{\lb^{\cal P}}{Z}\xra{\iso} \shf{Z}$.
\end{remark}

To simplify the notation (assuming that   
$\chi$ is clear from the context) we will just write $Z$-frame of ${\cal P}$, or framing of ${\cal P}$ along $Z$. A $Z$-framed $G$-bundle over $X$ is  a pair $({\cal P},\tau)$ consisting of a holomorphic principal $G$-bundle on $X$, and a $Z$-frame of ${\cal P}$. One has an obvious notion of isomorphism between $Z$-framed $G$-bundles. We will denote by $\mathrm{Bun}_G(X,Z,\chi)$ the set of isomorphism classes of $Z$-framed $G$-bundles on $X$.

We define the subsheaf ${\cal G}_Z^\chi\subset {\cal G}$ by
$${\cal G}_Z^\chi(U):=\{f\in {\cal G}(U)|\ \resto{(\chi\circ f)}{U\cap Z}=1\}\,.
$$
The $G$-bundle ${\cal P}_\alpha$ associated with a \v{C}ech cocycle $\alpha\in Z^1({\cal U},{\cal G}_Z^\chi)$ (for an open covering ${\cal U}$ of $X$) comes with a line bundle isomorphism 
$$\tau_\alpha: \resto{{\cal P}_\alpha\times_\chi \C}{Z}\xra{\simeq} Z\times\C\,.$$
The point is the line bundle $\resto{{\cal P}_\alpha\times_\chi \C}{Z}$ is defined by the cocycle 
$$(\chi\circ\alpha)_Z\in Z^1({\cal U}\cap Z,{\cal G}_Z^\chi)$$
which is trivial. It is easy to see that the isomorphism class of the $Z$-framed bundle $(Z_\alpha,\tau_\alpha)$ depends only on the cohomology class of $\alpha$; so gets a well-defined map
$$F:H^1(X,{\cal G}_Z^\chi)\to \mathrm{Bun}_G(X,Z,\chi)
$$
\begin{proposition}\label{framez}
Suppose that $\chi$ has a right inverse, i.e. that there exists a group morphism $\sigma:\C^*\to G$ such that $\chi\circ\sigma=\id$. Then the map $F$ is bijective.\end{proposition}
\begin{proof}
For the surjectivity, let $({\cal P},\tau)$ be a $Z$-framed bundle on $X$.
Any local trivialization $\theta:P_U\to U\times G$ induces a local
trivialisation 
$$\chi(\theta):P_U\times_\chi \C \ra U\times \C\,.$$
Define a sheaf of sets $\Tau$ on $X$ by
$$\Tau(U)=\{\theta:P_U\xra{\simeq} U\times G |\ \theta \text{ is a trivialisation}\}\,.$$
Note that for any $x\in X$  stalk ${\cal T}_x$ is a ${\cal G}_x$-torsor. Furthermore we define the subsheaf $\Tau'\subset \Tau$ of local trivialisations which are compatible with $\tau$:
$$\Tau'(U)=\{\theta \in \Tau(U) |\ \resto{\chi(\theta)}{U\cap
  Z}=\resto{\tau}{U\cap Z}\}\,.$$
%
For  $\theta \in \Tau(U)$ and $\zeta \in \shf{X}^*(U)$ we can define (using $\sigma$)   a trivialisation $\zeta\theta\in \Tau(U)$ by $\zeta\theta:=(\id_U, L_{\sigma(\zeta)})\circ \theta$, where $L_{\sigma(\zeta)}:U\times G\to   G$ is defined by
$$L_{\sigma(\zeta)}(u,g):= \sigma(\zeta(u))g\,.
$$
One has the identity
$$\chi(\zeta\theta)=\zeta(\chi(\theta))\,.
$$
Using this result, and the surjectivity of ${\cal O}^*_X\to {\cal O}^*_{Z,x}$
(for $x\in Z)$ it follows that for any $x\in Z$ the stalk $\Tau'_x$ is
non-empty. More precisely, this stalk is a ${\cal G}^\chi_{Z,x}$-torsor. For
ever $x\in Z$ let $U^x$ be an open neighbourhood of $x$, and $\theta^x\in {\cal
  T}'(U_x)$. Let also $(\theta^i:P_{U^i}\to U_i\times G)_{i\in I}$ be a system
of trivialisations of the restriction $\resto{P}{X\setminus Z}$. It suffices to
note that $(\theta^j)_{j\in Z\cup I}$ is a system of trivialisations of $P$, and
the associated cocycle $\alpha\in Z^1((U^j)_{j\in Z\cup I},{\cal G})$ belongs to
$Z^1((U^j)_{j\in Z\cup I},{\cal G}^\chi_Z)$. 

\end{proof}


A special case of Proposition \ref{framez} concerns the case $G=\C^*$, $\chi=\id_{\C^p*}$, i.e. the case of holomorphic line bundles on $X$ framed along $Z$.  Denote $\Pic(X,Z)$ the set of isomorphism classes of
  pairs $(\lb,\tau)$ consisting of a line bundle $\lb$ on $X$ and a trivialization $\tau$ of $\resto{\lb}{Z}$.   Define the sheaf $_Z\shf{X}^*$ to be the subsheaf of $\shf{X}^*$ that consists of germs of non-vanishing
holomorphic functions that are
equal to $1\in\cstar$ for any point in $Z$. Proposition \ref{framez} yields:
\begin{proposition}
 There is a natural identification $\Pic(X,Z)\iso H^1(X,{_Z\shf{X}^*})$.
\end{proposition}

\subsection[Appendix: A representability theorem]{Appendix 2}\label{functorstuff}
Let ${\cat C}$ be a category and $F:\cat{C} \ra \Set$ be contravariant functor.
Recall that $F$ is called representable if there exists an object $U\in {\cal
  O}b({\cat C})$ and a natural isomorphism $F \simeq \Hom(\cdot,U)$. A
representation of $F$ is a pair $(U,\Phi)$ where $\Phi:F\xra{\simeq} \Hom(\cdot, U)$ is
a natural isomorphism. Recall that a universal element of $F$ is a pair
$(U,\alpha)$ consisting of an object $U\in {\cal O} b(\cat{C})$ and an element
$\alpha\in F(U)$ such that for every pair $(X,\beta)$ with $\beta \in F(X)$
there exists a unique morphism $f\in \Hom_{\cat{C}}(X,V)$ such that
$F(f)(\alpha)=\beta$.
\begin{remark}
  Universal elements of $F$ are in a bijective correspondance with
  representations of $F$.
\end{remark}
Let now ${\cat A}$ be a  category, let $P\in {\cal O}b({\cat A})$, and let ${\cat
  A}/P$ be the category whose objects are morphisms $T\to P$ with $T\in{\cal
  O}b({\cat A})$ and whose morphisms are commutative triangles. Let $B:({\cat
  A}/P)^{op}\to {\Set }$ be a contravariant functor, and let $\tilde B:{\cat
  A}^{op}\to \Set$ the contravariant functor defined on objects by 
$$\tilde B(T):=\{(f,\beta)|\ f\in \Hom(T,P),\ \beta\in B(f)\}
$$
and on morphisms $u:S\to T$ by 
$$\tilde B(u)(f,\beta):=(f\circ u, B(u_f)(\beta))\,,
$$
where $u_f\in\Hom(f\circ u, f)\in\Hom_{{\cat A}/P}(f\circ u, f)$ is the morphism
associated with $u$.

\begin{proposition}\label{app:rep}
Let $V\in {\cal O}b(\cat{A}), v\in \Hom(V,P)$ and $\alpha\in
  B(V\xra{v} P)$. Then $(V\xra{v} P, \alpha)$ is a universal element of $B$ if
  and only if $(V, (v,\alpha))$ is a universal element of ${\tilde B}$.
  In particular, $B$ is representable if and only if ${\tilde B}$ is representable.
\end{proposition}

\begin{proof}
  Suppose that $(V\xra{v} P, \alpha)$ is a universal element of $B$ and let
  $(f,\beta)\in {\tilde B}(T)$. Then there exists a unique morphism $g\in \Hom_P(T_f,
  V_v)$ such that $B(g)(\alpha)=\beta$. But that just means that ${\tilde
    B}(g)(v,\alpha)=(f,\beta)$. Clearly $g$ is unique as a morphism in
  $\Hom(T,V)$ so $(V,(v,\alpha))$ is a universal element of ${\tilde B}$. The
  other direction of the proof follows from the same argument.
\end{proof}

\bibliographystyle{alpha}
\bibliography{representation.bib}
\end{document}